\documentclass[12pt]{amsart}
\usepackage{amsfonts,amscd,latexsym,amsmath,amssymb,cancel,enumerate}
\theoremstyle{plain}
\newtheorem{master}{Master}[section]
\newtheorem{prop}[master]{Proposition}
\newtheorem{thm}[master]{Theorem}
\newtheorem{fact}[master]{Fact}
\newtheorem{lem}[master]{Lemma}
\newtheorem{cor}[master]{Corollary}

\newtheorem{claim}[master]{Claim}
\theoremstyle{definition}
\newtheorem{defin}[master]{Definition}
\newtheorem{observation}[master]{Observation}
\theoremstyle{remark}

\numberwithin{equation}{section}
\begin{document}
\title{$F_\sigma$ equivalence relations and Laver forcing}
\author{Michal Doucha}
\address{Institute of Mathematics, Academy of Sciences, Prague, Czech republic}
\email{m.doucha@post.cz}
\thanks{The research of the author was partially supported by NSF grant DMS 0801114 and grant IAA100190902 of Grant Agency of the Academy of Sciences
of the Czech Republic}
\keywords{Borel equivalence relations, Laver ideal, Canonical Ramsey theorem.}
\subjclass[2000]{03E15, 03E05.}
\begin{abstract}
Following the topic of the book Canonical Ramsey Theory on Polish Spaces by V. Kanovei, M. Sabok and J. Zapletal we study Borel equivalences on Laver trees. Here we prove that equivalence relations Borel reducible to an equivalence relation on $2^\omega$ given by some $F_\sigma$ $P$-ideal on $\omega$ can be canonized to the full equivalence relation or to the identity relation.

This has several corollaries, e.g. Silver type dichotomy for the Laver ideal and equivalences Borel reducible to equivalence relations given by $F_\sigma$ $P$-ideals.
\end{abstract}
\maketitle
\section*{Introduction}
The aim of this paper is to prove canonization results for the Laver ideal in the spirit of \cite{KSZ}.

Let us recall that a Borel equivalence relation $E$ on a standard Borel space $X$ is an equivalence relation which is a Borel subset of $X\times X$. All equivalence relations in this text will be assumed to be Borel. We say that an equivalence relation $E$ on $X$ is Borel reducible to an equivalence $F$ on $Y$, $E\leq _B F$, if there exists a Borel function $f:X\rightarrow Y$ such that $xEy\equiv f(x)Ff(y)$. They are bireducible, $E\approx_B F$, if $E\leq _B F$ and $F\leq _B E$. For a Borel subset $A\subseteq X$, $E\upharpoonright A$ is the Borel equivalence relation $E\cap A\times A$, the restriction of $E$ on $A$.

For a Borel ideal $\mathcal{I}$ on $\omega$ we denote $E_\mathcal{I}$ the Borel equivalence relation on $2^\omega$ where $xE_\mathcal{I}y\equiv \{n:x(n)\neq y(n)\}\in \mathcal{I}$. Obviously, if $\mathcal{I}$ is $\Sigma^0_\alpha$ ($\Pi^0_\alpha$), then also $E_\mathcal{I}$ is $\Sigma^0_\alpha$ ($\Pi^0_\alpha$).

An ideal $\mathcal{I}$ on $\omega$ is called $P$-ideal if $\forall (A_n)_{n\in \omega}\subseteq \mathcal{I}\exists A\in \mathcal{I}\forall n (A_n\subseteq_* A)$.

Let us recall that a Laver tree $T\subseteq \omega^{<\omega}$ is a tree with stem $s$, the maximal node such that every other node is compatible with it, such that every node above $s$ (and including $s$) splits into infinitely many immediate successors. The set of all branches of $T$ is denoted as $[T]$. 

We can now state the main result of this paper.
\begin{thm}
Let $T$ be a Laver tree, $\mathcal{I}$ an $F_\sigma$ $P$-ideal on $\omega$ and $E\subseteq [T]\times [T]$ be an equivalence relation Borel reducible to $E_\mathcal{I}$. Then there is a Laver subtree $S\leq T$ such that $E\upharpoonright [S]$ is either $\mathrm{id}([S])$ or $[S]\times [S]$.
\end{thm}
We note that the subtree $S$ in general cannot be found as a direct extension of $T$.

The list of equivalence relations Borel bireducible with $E_\mathcal{I}$ for $\mathcal{I}$ an $F_\sigma$ $P$-ideal includes for instance $E_{\ell_p}$ equivalences for $p\in [1,\infty)$ on $\mathbb{R}^\omega$, where $xE_{\ell_p}y\equiv x-y\in \ell_p$; or $E_2$ (which is in fact bireducible with $E_{\ell_1}$) on $2^\omega$, where $xE_2y\equiv \sum \{1/(n+1):x(n)\neq y(n)\}<\infty$.
\section{Preliminaries}
To get into the context, we now present the general program of \cite{KSZ}: let $X$ be a Polish space, $I$ a $\sigma$-ideal on $X$ and $E\subseteq X^2$ a Borel equivalence relation.
\begin{itemize}
\item We say that $E$ is in the spectrum of $I$ if there exists a Borel set $B\in I^+$ such that $\forall C\in (I^+\cap \mathrm{Borel}(B))$ $E\upharpoonright C$ has the same complexity as $E$ on the whole space, i.e. $E\upharpoonright C$ is Borel bireducible with $E\upharpoonright X$.
\item On the other hand, $I$ canonizes $E$ to a relation $F\leq _\mathrm{B} E$ if for every Borel $B\in I^+$ there is some Borel $C\in (I^+\cap \mathrm{Borel}(B))$ such that $E\upharpoonright C$ is bireducible with $F$.
\end{itemize}   

Before proving the main theorem we state existing knowledge about the spectrum of Laver ideal and some results about $F_\sigma$ $P$-ideals and Laver ideal that we will need in the proof of the theorem.

The following theorem connects $F_\sigma$ $P$-ideals with lower semicontinuous submeasures. Recall that a submeasure $\mu: \mathcal{P}(\omega)\rightarrow [0,\infty]$ is lower semicontinuous if it is lower semicontinuous in the Cantor space topology on $\mathcal{P}(\omega)$ (identified with $2^\omega$).
\begin{thm}[Solecki \cite{So}]\label{Sol}
Let $\mathcal{I}$ be an $F_\sigma$ $P$-ideal on $\omega$. Then there exists a lower semicontinuous submeasure $\mu :\mathcal{P}(\omega)\rightarrow [0,\infty]$ such that $\mathcal{I}=\mathrm{Exh}(\mu)=\mathrm{Fin}(\mu)$, where $\mathrm{Exh}(\mu)=\{A\subseteq \omega:\lim _{n\to \infty} \mu(A\setminus n)=0\}$ and $\mathrm{Fin}(\mu)=\{A\subseteq \omega:\mu(A)<\infty\}$.
\end{thm}
We add some notation concerning Laver trees and Laver ideal. We say that a Laver tree $S$ is a direct extension of a Laver tree $T$, $S\leq_0 T$ in symbols, if the stem of $S$ is the same as the stem of $T$. If $s\in T$ is a node above the stem then by $T_s$ we denote the induced subtree with $s$ as the stem, i.e. $T_s=\{t\in T: t\text{ is compatible with }s\}$.

We use the definition of Laver ideal $I$ from \cite[p.200]{Zap}; $I\subseteq \mathcal{P}(\omega^\omega)$ is the $\sigma$-ideal generated by sets $A_g=\{f\in \omega^\omega:\exists^\infty n (f(n)\in g(f\upharpoonright n))\}$, where $g$ is a function from $\omega^{<\omega}$ to $\omega$.

The following proposition, resp. its corollary, will be used extensively. The proof is in \cite{Zap}.
\begin{prop}\label{laverfactmain}
Let $A\subseteq \omega^\omega$ be analytic. Then either $A$ contains all branches of some Laver tree or $A\in I$.
\end{prop}
We will provide a proof of the following corollary. Recall that a barrier $B$ in a Laver tree $T$ is a subset of nodes such that $\forall x\in [T]\exists n (x\upharpoonright n\in B)$.
\begin{cor}\label{laverfact}
Let $T$ be a Laver tree and let $A\subseteq [T]$ be analytic. Then there exists a direct extension $S\leq_0 T$ such that either $[S]\subseteq A$ or $[S]\cap A=\emptyset$.
\end{cor}
\begin{proof}
It follows from the proposition above that there is always $S\leq T$ with that property which is in general not a direct extension though. The use of "direct extension property" will give us the desired tree. Let $t$ be the stem of $T$. If there exist infinitely many immediate successors $s$ of $t$ such that there exists a direct extension $S\leq_0 T_s$ with the property above, then for infinitely many of them it holds that $[S]\subseteq A$, or for infinitely many of them it holds that $[S]\cap A=\emptyset$, and we use them.  So suppose that not, we erase these finitely many exceptions and proceed to the next level and do the same. At the end we obtain a Laver tree $T'\leq_0 T$. We apply the proposition above and get a node $t\in T'$ and a direct extension $S\leq_0 T'_t$ such that either $[S]\subseteq A$ or $[S]\cap A=\emptyset$. That is a contradicition since such a node was erased during the construction of $T'$.
\end{proof}

We now state some results that have been obtained about the spectrum of Laver in \cite{KSZ}. We will need just Corollary \ref{ctble} in our proof, however we state the general canonization result from which it follows.
\begin{thm}[\cite{KSZ}]
Let $I$ be a $\sigma$-ideal on a Polish space $X$ such that the quotient forcing $P_I$ is proper, nowhere ccc and adds a minimal forcing extension. Then $I$ has total canonization for equivalence relations classifiable by countable structures.
\end{thm}
As Laver ideal fulfils these conditions we immediately get the following corollary.
\begin{cor}[\cite{KSZ}]
Let $T$ be a Laver tree, $E$ an equivalence classifiable by countable structures. Then there is a Laver subtree on which $E$ is either the identity relation or the full relation.
\end{cor}
\begin{cor}[\cite{KSZ}]\label{ctble}
Let $T$ be a Laver tree, $E$ a countable equivalence relation (i.e. with countable classes). Then there is a Laver subtree on which $E$ is either the identity relation or the full relation.
\end{cor}

J. Zapletal found the following $F_\sigma$ equivalence relation (with $K_\sigma$ classes) that is in the spectrum of Laver.
\begin{defin}\label{defofK}
For $x,y\in \omega^\omega$, we set $xKy$ if $\exists b \forall n \exists m_x,m_y\leq b (y(n+m_y)\geq x(n)\wedge x(n+m_x)\geq y(n))$.\\
\end{defin}
The following lemma gives us basic properties of $K$. The proof may be found in \cite{KSZ}, we provide here the proof of the last item as it is stated slightly differently in \cite{KSZ}. Notice the difference between $E_{\ell_p}$ for $p\in [1,\infty)$ and $E_{\ell_\infty}$ as the former can be canonized according to the main theorem.
\begin{lem}
\emph{}
\begin{enumerate}[(a)]
\item For any two Laver trees $T,S$ there are branches $x_1,x_2\in [T]$ and $y_1,y_2\in [S]$ such that $x_1Ky_1$ and $x_2\cancel{K}y_2$.
\item $K$ is in the spectrum of Laver.
\item $K$ is Borel bireducible with $E_{\ell_\infty}\subseteq \mathbb{R}^\omega\times\mathbb{R}^\omega$, where $xE_{\ell_\infty}y\equiv x-y\in \ell_\infty$.
\end{enumerate}
\end{lem}
\begin{proof}
\emph{}
\begin{enumerate}[(a)]
\item 
\item We refer to \cite{KSZ} for the proof of the first two items.
\item \begin{itemize}
\item $E_{\ell_\infty}\leq _B K$: We will prove $E_{\ell_\infty}\leq _B E_{\ell_\infty}\upharpoonright (\mathbb{R}^+)^\omega\leq _B K$. To prove the first inequality, consider $f:\mathbb{R}^\omega\rightarrow (\mathbb{R}^+)^\omega$ such that if $x(n)\geq 0$ then $f(x)(2n)=x(n)$ and $f(x)(2n+1)=0$ and if $x(n)<0$ then $f(x)(2n)=0$ and $f(x)(2n+1)=|x(n)|$.

For the second, let $\pi:\omega^2\rightarrow \omega$ be a bijection and $(I_{\pi(i,j)})_{i,j}$ a partition of $\omega$ into intervals such that $|I_{\pi(i,j)}|=j+1$ and $I_{\pi(i,j)}=\{p_0^{i,j},p_1^{i,j},\ldots,p_j^{i,j}\}$. We define $g:(\mathbb{R}^+)^\omega\rightarrow \omega^\omega$ as follows: Let $x\in (\mathbb{R}^+)^\omega$ be given, $g(x)(p_k^{i,j})=\min \{j,\lfloor x(i)\rfloor+k\}$ for $k<j$ and $g(x)(p_j^{i,j})=j$.

If $xE_{\ell_\infty}y$ and $\forall n (|x(n)-y(n)|\leq m)$, then $\forall k,ij\exists m_1,m_2\leq m (g(x)(p_k^{i,j})\leq g(y)(p_{k+m_1}^{i,j})\wedge g(y)(p_k^{i,j})\leq g(x)(p_{k+m_2}^{i,j}))$, thus $g(x)Kg(y)$. Just observe that $g(x)(p_k^{i,j})\leq j$ and either $j-k\leq m$ and we have $g(y)(p_j^{i,j})=j$, or since $x(i)-y(i)=m_1\leq m$ we have $g(y)(p_{k+m_1}^{i,j})=y(i)+k+m_1=x(i)+k=g(x)(p_k^{i,j})$.

Suppose $x\cancel{E}_{\ell_\infty}y$, let $m$ be arbitrary and let $n$ be such that $|x(n)-y(n)|>m$, let us assume that $y(n)-x(n)>m$. Then $\forall b\leq m (g(x)(p_{k+b}^{n,m})<g(y)(p_k^{n,m}))$. Since $m$ was arbitrary we have $g(x)\cancel{K}g(y)$.
\item $K\leq _B E_{\ell_\infty}$: Let $(s_n)_n$ be an enumeration of $\omega^{<\omega}$. We define $f: \omega^\omega\rightarrow \mathbb{R}^\omega$ as follows: $f(x)(n)=\min\{b: \exists y\supseteq s_n (xKy\wedge b\text{ is the bound from the definition that works})\}$. One can easily check that $f$ is Borel. Let $xKy$ such that a bound $b$ works for this pair and let $n$ be arbitrary. Let $z\supseteq s_n$ be arbitrary such that $xKz$ and $b_1$ works for the pair and $yKz$ and $b_2$ works for the pair. Then one can check that $|b_1-b_2|\leq b$ so $f(x)E_{\ell_\infty}f(y)$.

Suppose that $x\cancel{K}y$ and let $m$ be arbitrary. Then there exists $n$ such that $x(n)>y(n+k)$ for $k<m$ (or vice versa). Let $s_i=x\upharpoonright (n+1)$, then $f(x)(n)=0$, however $f(y)(n)\geq m$, thus $f(x)\cancel{E}_{\ell_\infty}f(y)$.

\end{itemize}

\end{enumerate}
\end{proof}
\section{Main result}
We can now start proving the main theorem, we provide its statement here again for the convenience.
\begin{thm}\label{main}
Let $T$ be a Laver tree, $\mathcal{I}$ an $F_\sigma$ $P$-ideal on $\omega$ and $E\subseteq [T]\times [T]$ be an equivalence relation Borel reducible to $E_\mathcal{I}$. Then there is a Laver subtree $S\leq T$ such that $E\upharpoonright [S]$ is either $\mathrm{id}([S])$ or $[S]\times [S]$.
\end{thm}
\begin{proof}
Let $f:[T]\rightarrow 2^\omega$ be the Borel reduction and let $\mu$ be the lower semicontinuous submeasure for $\mathcal{I}$ guaranteed by Theorem \ref{Sol}. The submeasure $\mu$ induces a pseudometrics (which may attain infinite value though) which we denote $d$, i.e. $d(x,y)=\mu(x\bigtriangleup y)$ for $x,y\in 2^\omega$. Moreover, we define $d_n^k(x,y)$ as $\mu(x\upharpoonright (n,k)\bigtriangleup y\upharpoonright (n,k))$. When $n$ or $k$ is omitted it means that $n=0$, resp. $k=\infty$.

We need to refine $T$ to obtain a Laver tree with some special properties. This will be done in a series of claims. To simplify the notation, after applying each one of these claims we will still denote the tree as $T$.
\begin{claim}
There exist a direct extension $T'\leq_0 T$ and a function $p:T'\rightarrow 2^{<\omega}$ which is monotone and preserves length of sequences, i.e. if $s\subseteq t$, then $p(s)\subseteq p(t)$, and $|s|=|p(s)|$, such that $\forall x\in [T'] (f(x)=\bigcup_n p(x\upharpoonright n))$.

In other words, $f$ on $[T']$ is Lipschitz. 
\end{claim}
\emph{Proof of the Claim}. We will find a direct extension of $T$ and $p$ defined on it from the statement of the claim. For simplicity we assume the stem of $T$ is the empty sequence.

Consider the following sets $$A_i=\{x\in [T]: f(x)(0)=i\}$$ for $i\in \{0,1\}$. They are Borel and according to Corollary \ref{laverfact} one of them contains a direct extension $S$ of $T$. We replace $T$ by $S$, set $p(\emptyset)=i$ and fix the first level above the stem. Then for any immediate successor $s$ of the stem we again consider sets $A^1_i=\{x\in [S_s]: f(x)(1)=i\}$. One of them contains direct extension and we continue similarly. The final tree is obtained by fusion.$\qed$\\

\begin{observation}\label{observ}
Let $s\in T$ be a node above (or equal to) the stem of $T$. Then for every $n$ there is a direct extension $T^n_s\leq_0 T_s$ such that $\forall x,y\in [T^n_s] \forall m\leq n (p(x\upharpoonright m)=p(y\upharpoonright m))$. We will call such a tree homogeneous up to level $n$.

We may also suppose that we have $T^n_s\subseteq T^m_s$ for $n\geq m$. Define then $x_s\in 2^\omega$ such that $x_s(n)=p(x\upharpoonright n+1)(n)$ for $x\in [T^m_s]$, where $m\geq n+1$. This definition does not depend on $m\geq n+1$ and $x\in [T^m_s]$.
\end{observation}
The following can be done by a basic fusion argument.
\begin{fact}
There exists a direct extension $T'\leq_0 T$ such that $\forall s\in T'$ above the stem if $S\leq_0 T'_s$ is homogeneous up to some level $n$ then $\forall x\in [S] (p(x\upharpoonright n)=x_s\upharpoonright n)$.

Moreover, if $s\in S$ and $\{s_0,s_1,\ldots\}$ is a set of its immediate successors then $\forall n\exists m_0\forall m\geq m_0 (x_s\upharpoonright n=x_{s_m}\upharpoonright n)$. In other words, $\lim_{n\to \infty} x_{s_n}\to x_s$.
\end{fact}
Let $s\in T$ be any node above (or equal to) the stem of $T$. Let $\{s_0,s_1,\ldots\}$ be the set of its immediate successors. We reduce this set so that precisely one of the following two possibilities happens: $\forall n (x_{s_n}E_\mathcal{I}x_s)$ or $\forall n (x_{s_n}\cancel{E}_\mathcal{I}x_s)$.
\begin{defin}
If the former case holds then we mark $s$ as "convergent", if the latter then we mark it as "divergent".

Moreover, for every $s\in S$ strictly above the stem we define $\varepsilon_s$ as follows: if the immediate predecessor $t$ of $s$ is marked as convergent, then we set $\varepsilon_s=d(x_t,x_s)$; otherwise, we set $\varepsilon_s=\infty$.
\end{defin}
\vspace{4mm}
{\bf Splitting into cases}\\
We split into two complementary cases (i.e. one holds if and only if the other does not).
\begin{itemize}
\item {\bf Case 1} There exists $S\leq T$ such that every $s\in S$ above the stem is marked as convergent.\\
\item {\bf Case 2} For every $s\in T$ above the stem there is a barrier $B\subseteq T_s$ of elements above $s$ that were marked as divergent.\\
\end{itemize}\vspace{3mm}
\emph{Proof of canonization assuming} Case 1. We will do a fusion. Let us denote the stem of $S$  as $s$. We will inductively build $U_n,S_n,m_n$ for every $n$ such that $S_n\leq_0 S_{n-1}$, $U_n\subseteq S_m$, for every $n\leq m$, is an $n+1$-element subtree $\{u_0,\ldots,u_n\}$ of $S$ and $m_n\in \omega$. At the end we will get a direct extension $U=\bigcup _n U_n=\bigcap _n S_n$ together with pairwise disjoint sets $C_{u_1},C_{u_2},\ldots$, where $C_{u_i}\subseteq (m_{i-1},m_i)$, such that $\forall x\in [U]$ $(f(x)\bigtriangleup x_s)\cap (m_{i-1},m_i)=C_{u_i}$ if $u_i\subseteq x$ and $\mu(\bigcup_ {\{i>0:u_i\nsubseteq x\}} f(x)\cap (m_{i_1},m_i))<1$. The following conditions will be satisfied during the $n$-th step of the fusion.
\begin{itemize}
\item For every $0<i\leq n$ and any branch $x\in [S_n]$ going through $u_i$ $|d_{m_{i-1}}^{m_i}(f(x),x_s)-\varepsilon_{u_i}|<1/2^i$; more precisely there will be some finite set $C_{u_i}\subseteq (m_{i-1},m_i)$ always defined as $(x_{u_i}\bigtriangleup x_s)\cap (m_{i-1},m_i)$ such that for any branch $x\in [S_n]$ going through $u_i$ we will have $(f(x)\bigtriangleup x_s)\cap (m_{i-1},m_i)=C_{u_i}$ and $|\mu(C_{u_i})-\varepsilon_{u_i}|<1/2^i$. And for every branch $y\in [U_n]$ not going through $u_i$ but going through some other $u_j$,  $d_{m_{i-1}}^{m_i}(f(y),x_s)<1/2^i$; thus it will follow from the triangle inequality that $|d_{m_{i-1}}^{m_i}(f(x),f(y))-\varepsilon_{u_i}|<1/2^{i-1}$; resp. $\mu((f(x)\bigtriangleup f(y)\bigtriangleup C_{u_i})\cap (m_{i-1},m_i))<1/2^i$.
\item For every $i\leq n$ $d_{m_n}(x_{u_i},x_s)<1/2^{n+2}$.
\end{itemize}
 
Suppose at first that such $U$ has been already constructed. Let us consider the set $$A=\{x\in [U]: \mu (\bigcup_{i=|s|+1}^\infty C_{x\upharpoonright i})<\infty\}$$ It is Borel and by Corollary \ref{laverfact} either there is a Laver subtree $V\leq_0 U$ such that $[V]\subseteq A$ or there is a Laver subtree $V\leq_0 U$ such that $[V]\cap A=\emptyset$. In the former case, $V$ is a Laver subtree such that $\forall x,y\in [V] (xEy)$; while in the latter case, $V$ is a Laver subtree such that $\forall x,y\in [V] (x\cancel{E}y)$. This follows immediately from the condition above. Let $x,y\in [V]$ be two different branches splitting on the $n$-th level. Then $\max\{\mu(\bigcup_{i=n}^\infty C_{x\upharpoonright i}),\mu(\bigcup_{i=n}^\infty C_{y\upharpoonright i})\}-\sum _{j=n-|s|+1}^\infty 1/2^j\leq d(f(x),f(y))\leq \mu(\bigcup_{i=n}^\infty C_{x\upharpoonright i})+\mu(\bigcup_{i=n}^\infty C_{y\upharpoonright i})+\sum _{j=n-|s|+1}^\infty 1/2^{j-1}$.\\

Let $s$ be the stem of $S$. Set $S_0=S$, $U_0=\{s\}$, $m_0=|s|$. Before treating the general step let us describe the case $n=1$. We pick some immediate successor of the stem $s$, denote it as $u_1$ and we set $U_1=\{s=u_0,u_1\}$. Since $d(x_{u_1},x_s)=\varepsilon_{u_1}$ there is some $m>m_0$ such that $d^m(x_{u_1},x_s)>\varepsilon_{u_1}-1/2$. There is some $m_1\geq m$ such that $d_{m_1}(x_{u_1},x_s)<1/2^3$. Then there exist direct extensions $E_1\leq_0 S_{0 u_1}$ and $E_0\leq_0 S_0$ such that for all branches $x\in [E_1]$ we have $f(x)(m)=x_{u_1}(m)$ for $m\leq m_1$, and for all branches $y\in [E_0]$ we have $f(y)(m)=x_s(m)$ for $m\leq m_1$. We set $S_1=E_0\cup E_1$, i.e. we replace $S_{0 u_1}$ in $S_0$ by its direct extension $E_1$ and we replace $S_0\setminus S_{0 u_1}$ by its direct extension $E_0$. The required conditions are satisfied and we proceed to a general step.

Now let us suppose that we have already found $S_{n-1}$,$U_{n-1}=\{s=u_0,u_1,\ldots,u_{n-1}\}$ and $m_0=|s|,m_1,\ldots,m_{n-1}$. Choose some next node $u_n\in S_{n-1}$ for the fusion such that it is an immediate successor of some $u_i$ $i<n$ and $x_{u_i}\upharpoonright m_{n-1}=x_{u_n}\upharpoonright m_{n-1}$ (recall Observation \ref{observ}). Set $U_n=U_{n-1}\cup \{u_n\}$. There is some $m$ such that $d_{m_{n-1}}^m(x_{u_n},x_{u_i})>\varepsilon_{u_m}-1/2^{n+1}$. Since we have from the inductive assumption that $d_{m_{n-1}}^m(x_{u_i},x_s)<1/2^{n+1}$, we get from the triangle inequality $d_{m_{n-1}}^m(x_{u_n},x_s)>\varepsilon_{u_m}-1/2^n$. Let $m_n$ be the $\max\{m,\max\{k_i:i\leq n\}\}$, where $k_i$ is any number such that $d_{k_i}(x_{u_i},x_s)<1/2^{n+2}$. Note that such $k_i$ exist because $x_{u_i}E_\mathcal{I}x_s$. We then find direct extensions $E_i\leq_0 S_{n-1,u_i}$ such that for every branch $x\in [E_i]$ we have $f(x)(m)=x_{u_i}(m)$ for $m\leq m_n$. We refine them so that they are mutually disjoint, i.e. $E_i\cap E_j=\emptyset$ for $i\neq j$ and we set $S_n=\bigcup _{i\leq n} E_i$. The induction step is done, all required conditions are satisfied. That finishes the proof of this case.
\vspace{3mm}\\
\emph{Proof of canonization assuming} Case 2. We will assume that we have a Laver tree $S\leq T$ such that for every $s\in S$ above the stem if $s$ is marked as convergent then there is a barrier $B\subseteq S_s$ of elements above $s$ that are marked as divergent (if we assume that Case 1 does not hold then we may take $S=T$).

The following lemma will be the main tool.
\begin{lem}
For any Laver subtree $P\leq S$ there is its direct extension $Q\leq_0 P$ such that for any two branches $x,y\in [Q]$ splitting from the stem of $Q$ we have $d(f(x),f(y))>1$.
\end{lem}
Once the lemma is proved the rest will be rather easy. We will do a fusion in which we will be fixing levels. We will construct direct extensions of $S$ $S=V_0\geq _0 V_1\geq _0 V_2\geq _0\ldots$ such that for $i<j$ the $i$-th level of $V_i$ is equal to the $i$-th level of $V_j$ in such a way that the resulting tree $S\geq _0 V=\bigcap _i V_i$ will have the property that for any two different branches $x,y\in [V]$ we will have $d(f(x),f(y))>1$.

This is not hard to do. We start with stem $s$ of $S=V_0$. We find a direct extension $V_1\leq _0 V_0$ guaranteed by the lemma. We fix the first level $\{s_0,s_1,\ldots\}$ (the set of all immediate successors of $s$) above the stem. Then for every immediate successor $s_i\in V_1$ of $s$ we apply the lemma with $V_{1 s_i}$ as $P$ and obtain a direct extension $Q_i$. We set $V_2=\bigcup_i Q_i\leq_0 V_1$, fix the second level above the stem and continue similarly.

Then we are done by the following claim and Corollary \ref{ctble}.
\begin{claim}
$E$ on $[V]$ is countable.
\end{claim}
\begin{proof}
Suppose for contradiction that there is some $x\in [V]$ that has uncountably many equivalent branches $(y_\alpha)_{\alpha<\omega_1}\subseteq [V]$. For every $\alpha$ there is $n$ such that $d_n(f(x),f(y_\alpha))<1/2$. Since the set of all $y_\alpha$'s is uncountable we may assume that one single $n$ works for them all. But let $y_{\alpha_0},y_{\alpha_1}$ be two of such branches that split above the $n$-th coordinate. It follows that from our construction that $d_n(f(y_{\alpha_0}),f(y_{\alpha_1}))>1$, so for one of them, let us say $y_{\alpha_0}$, must hold that $d(f(x),f(y_{\alpha_0}))>1/2$, a contradiction.
\end{proof}
So what remains is to prove the lemma.\\\emph{}\\
\emph{Proof of the lemma}. Let $P\leq S$ be given. Denote $s$ its stem. There are two cases.
\begin{itemize}
\item $s$ is marked as divergent: Pick its immediate successor $s_0$. Since $s$ is marked as divergent, there is $n_0$ such that $d^{n_0}(x_{s_0},x_s)>1$ and there are direct extensions $Q_0\leq_0 P_{s_0}$, $P_0\leq_0 P\setminus P_{s_0}$ such that $\forall x\in [Q_0]\forall y\in [P_0] \forall m\leq n_0(f(x)(m)=x_{s_0}(m)\wedge f(y)(m)=x_s(m))$.

We then pick next immediate successor $s_1\in P_0$ of $s$. There is again some $n_1$ such that $d^{n_1}(x_{s_1},x_s)>1$ and we find direct extensions $Q_1\leq_0 P_{0 s_1}$, $P_1\leq_0 P_0\setminus P_{0 s_1}$ such that $\forall x\in [Q_1]\forall y\in [P_1] \forall m\leq n_1(f(x)(m)=x_{s_0}(m)\wedge f(y)(m)=x_s(m))$.

We continue similarly until we pick infinitely many immediate successors of $s$ and find corresponding direct extensions $Q_i$. Then we set $Q=\bigcup _i Q_i$. It is easy to check it has the required properties.
\item $s$ is marked as convergent: There is a barrier $B\subseteq P$ of elements that were marked as divergent. We may assume that for every $b\in B$ and every $s\leq t<b$, $t$ is marked as convergent. We will do a similar fusion to that in the proof of canonization assuming Case 1. We will inductively build $Q_n, P_n, m_n$ such that $P_n\leq_0 P_{n-1}$, $Q_n=(\{q_0=s,\ldots,q_n\}\cup R)\subseteq P_m$ for $n\leq m$, $m_n\in \omega$. Let $\{q_i:i\in C\}\subseteq \{q_0,\ldots,q_n\}$ be the (possibly empty) set of those elements that are immediate successors of some element from $B$. Then $R=\bigcup _{i\in C} P_{i,q_i}$. The final tree is again obtained as $Q=\bigcup _i Q_i=\bigcap_i P_i$. Conditions that must be satisfied during the $n$-th step of the fusion are the following.
\begin{itemize}
\item For every $i<n$ if $i\notin C$, i.e. $q_i$ is not an immediate successor of an element from $B$, then for any branch $x\in [P_n]$ going through $q_i$ we have $d_{m_{n-1}}^{m_n}(f(x),x_s)<1/2^n$. And if $n\in C$, i.e. $q_n$ is an immediate successor of an element from $B$, then for any branch $y\in [P_n]$ going through $q_n$ we have $d_{m_{n-1}}^{m_n}(f(y),x_s)>2$; thus it will follow from the triangle inequality that $d_{m_{n-1}}^{m_n}(f(x),f(y))>1$.
\item For every $i\leq n$ if $i\notin C$, i.e. $q_i$ is not an immediate successor of an element from $B$, then $d_{m_n}(x_{q_i},x_s)<1/2^{n+2}$.
\end{itemize}
Suppose at first that such $Q$ has been constructed. We need to prove that for any two branches $x,y\in [Q]$ with $s$ as the last common node we have $d(f(x),f(y))>1$. It follows from the assumption that $x$ goes through some $u_i$ which is an immediate successor of some element from $B$, similarly $y$ goes through some different $u_j$ with the same property. Assume $i<j$. Then we get from the inductive assumption that $d_{m_{i-1}}^{m_i}(f(x),f(y))>1$ and we are done.\\

In the first step of the induction we set $Q_0=\{s\}$, $P_0=P$ and $m_0=|s|$; the set $R$ is empty.

Suppose we have already found $Q_{n-1},P_{n-1},m_{n-1}$. We choose some $q_n$ that is an immediate successor of some $q_i$. We have two cases.
\begin{itemize}
\item $q_i\notin B$, i.e. $q_n$ is not an immediate successor of an element from $B$. Then we set $m_n=\max\{k_i:i\leq n,i\notin C\}$, where $k_i$, for $i\notin C$, is any number such that $d_{k_i}(x_{q_i},x_s)<1/2^{n+2}$. Note that such $k_i$ exist because $x_{u_i}E_\mathcal{I}x_s$. We then find direct extensions $E_i\leq_0 P_{n-1,q_i}$ for $i\leq n,i\notin C$ such that for every branch $x\in [E_i]$ we have $f(x)(m)=x_{q_i}(m)$ for $m\leq m_n$. We refine them so that they are mutually disjoint, i.e. $E_i\cap E_j=\emptyset$ for $i\neq j$ and we set $P_n=(\bigcup _{i\notin C} E_i)\cup R$.
\item $q_i\in B$, i.e. $q_n$ is an immediate successor of an element from $B$. We add $n$ to $C$. There is some $m$ such that $d_{m_{n-1}}^m(x_{q_n},x_{q_i})>2+1/2^{n+1}$ since $q_i\in B$ is marked as divergent. Since from the inductive assumption we have $d_{m_{n-1}}^m(x_{q_i},x_s)<1/2^{n+1}$ we get from the triangle inequality that $d_{m_{n-1}}^m(x_{q_n},x_s)>2$. We then set $m_n=\max\{m,\max\{k_i:i\leq n,i\notin C\}\}$, where $k_i$'s are defined exactly the same as in the first case. We then again find direct extensions $E_i\leq_0 P_{n-1,q_i}$ for $i=n$ and $i<n,i\notin C$ such that for every branch $x\in [E_i]$ we have $f(x)(m)=x_{q_i}(m)$ for $m\leq m_n$. We refine them so that they are mutually disjoint, i.e. $E_i\cap E_j=\emptyset$ for $i\neq j$. We add $E_n$ to $R$ and we set $P_n=(\bigcup _{i\notin C} E_i)\cup R$.

\end{itemize}
In both cases it is easy to check that all required conditions are satisifed.
\end{itemize}
\end{proof}
\section{Corollaries}
\begin{thm}
Let $E\subseteq \omega^\omega\times \omega^\omega$ be an equivalence relation containing $K$, i.e. $E\supseteq K$, which is Borel reducible to $E_\mathcal{I}$ for some $F_\sigma$ $P$-ideal. Then there exists a Laver large set contained in one equivalence class.
\end{thm}
Recall that $K$ was defined in Definition \ref{defofK}.
\begin{proof}
Consider the set $$X=\{x\in \omega^\omega: [x]_E \text{ contains all branches of some Laver tree}\}$$ We use Theorem \ref{main} to prove that $X$ is non-empty. Suppose it is empty, then by Theorem \ref{main} there exists a Laver tree $T$ such that $E\upharpoonright [T]=\mathrm{id}([T])$. However, there must be two branches $x,y\in [T]$ such that $xKy$ and since $K\subseteq E$, also $xEy$, a contradiction.

Thus $X$ is non-empty. We show that it is also $E$-equivalent, i.e. there is no pair $x,y\in X$ such that $x\cancel{E}y$. Suppose the contrary. Then $[x]_E$ contains all branches of some Laver tree $T_x$ and $[y]_E$ contains all branches of Laver tree $T_y$ and there are branches $b_x\in T_x$ and $b_y\in T_y$ such that $b_xKb_y$ and since $K\subseteq E$, also $b_xEb_y$, a contradiction.

So $X$ is a single equivalence class, containing all branches of some Laver tree $T$, and thus it is Borel. If it were not Laver large, then the complement would be a Borel Laver positive set, so by Proposition \ref{laverfactmain} it would contain all branches of some Laver tree $S$. But we would again have that there are a branch $x\in [T]\subseteq X$ and a branch $y\in [S]$ such that $xKy$, thus $xEy$, a contradiction. 
\end{proof}
\begin{thm}[Silver dichotomy - under "$\forall x\in \mathbb{R}(\omega_1^{L[x]}<\omega_1)$"]
Let $E\subseteq \omega^\omega\times \omega^\omega$ be an equivalence relation Borel reducible to $E_\mathcal{I}$ for $F_\sigma$ $P$-ideal $\mathcal{I}$. Then either $\omega^\omega=(\bigcup _{n\in \omega} E_n)\cup J$, where $E_n$ for every $n$ is an equivalence class of $E$ and $J$ is a set in the Laver ideal, or there exists a Laver tree $T$ such that $E\upharpoonright [T]=\mathrm{id}([T])$.
\end{thm}
This is just a combination of Theorem \ref{main} and the results from the section on Silver dichotomy from \cite{KSZ}. It is not known if the assumption "$\forall x\in \mathbb{R}(\omega_1^{L[x]}<\omega_1)$" is necessary. 
\begin{cor}[under the same assumption]
Let $E\subseteq \omega^\omega\times \omega^\omega$ be an equivalence relation Borel reducible to $E_\mathcal{I}$ for $F_\sigma$ $P$-ideal $\mathcal{I}$ and let $X\subseteq \omega^\omega$ be an arbitrary Laver-positive subset (not necessarily definable) such that $\forall x,y\in X (x\cancel{E}y)$. Then there exists a Laver tree $T$ such that $E\upharpoonright [T]=\mathrm{id}([T])$.
\end{cor}
\begin{proof}
Just use the Silver dichotomy from the previous theorem and notice that the first possibility cannot happen. If $\omega^\omega=(\bigcup _{n\in \omega} E_n)\cup J$ as in the statement of the previous theorem, then $X\setminus J$ is still not in the Laver ideal and is uncountable.
\end{proof}


\begin{thebibliography}{4}
\bibitem{Ka}
V. Kanovei, \emph{Borel equivalence relations. Structure and classification}, American Mathematical Society, 2008
\bibitem{KSZ}V. Kanovei, M. Sabok, J. Zapletal, \emph{Canonical Ramsey Theory on Polish Spaces, to appear in 2013}
\bibitem{So}
S. Solecki, \emph{Analytic ideals and their applications}, Ann. Pure Appl. Logic 99 (1999) , no. 1-3, 51-72
\bibitem{Zap}
J. Zapletal, \emph{Forcing idealized}, Cambridge University Press, Cambridge, 2008
\end{thebibliography}
\end{document}